\documentclass[12pt]{article}
\usepackage{amsmath}
\usepackage{amsthm}
\usepackage{amsfonts}
\usepackage{amssymb}
\usepackage{fancyhdr}
\usepackage{subfigure}
\usepackage{setspace}
\usepackage{calligra}
\usepackage{epsfig,bm,color}
\usepackage{mathtools}
\usepackage[toc,page]{appendix}
  \newtheorem{theorem}{Theorem}[section]
\newtheorem{df}[theorem]{Definition}

  \newtheorem{lemma}[theorem]{Lemma}

    \newtheorem{conjecture}[theorem]{Conjecture}
\theoremstyle{definition}
 \newtheorem{ex}[theorem]{Example}
   
\newtheorem{remark}[theorem]{Remark}
\newcommand{\ZZ}{\mathbb{Z}}
\newcommand{\NN}{\mathbb{N}}
\newcommand{\ord}{\mathrm{ord}}

\begin{document}

\title{On Products of Strong Skolem Starters}
\author{
\vspace{0.25in}
Oleg Ogandzhanyants\footnote{oleg.ogandzhanyants@mun.ca} \hspace{1cm} Margarita Kondratieva\footnote{mkondra@mun.ca} \hspace{1cm} Nabil Shalaby\footnote{nshalaby@mun.ca}  \\
 Department of Mathematics and Statistics \\
 Memorial University of Newfoundland \\
 St. John's, Newfoundland \\
 CANADA A1C 5S7
}
 \maketitle
\begin{abstract}
In 1991, Shalaby conjectured that any $\ZZ_{n}$, where $n\equiv1$ or $3\pmod{8},\ n\ge11$, admits a strong Skolem starter. In 2018, the authors explicitly constructed some infinite ``cardioidal'' families of strong Skolem starters. No other infinite families of these combinatorial designs were known to date.

Statements regarding the products of starters, proven in this paper give a new way of generating strong or skew Skolem starters of composite orders. This approach extends our previous result by generating new infinite families that are not cardioidal.

The products that we introduce in this paper are multi-valued binary operations 
 which produce 2-partitions of  the set $\ZZ^*_{nm}$  of integers modulo $nm$ without zero,   from a pair of 2-partitions of $\ZZ^*_n$ and $\ZZ^*_m$, where $n, m\ge 3$ are odd integers. We prove several remarkable properties of these operations applied to starters in $\ZZ_n$ and to some other combinatorial objects that are 2-partitions of $\ZZ^*_n$ with additional restrictions, such as strong, skew, Skolem and cardioidal 2-partitions.

\end{abstract}

\section{Introduction}\label{introduction}

\indent A \textit{starter}  in an additive abelian group $G$ of odd order $n$ is a partition of the set $G^*$ of all non-zero elements of $G$ into $q=(n-1)/2$ pairs $\{\{s_i,t_i\}\}_{i=1}^q$ such that the elements $\pm(s_i-t_i), i=1,...,q$, comprise $G^*$.

Starters exist in any additive abelian group of odd order $n\ge3$. For example, the partition $\{\{x,-x\}\mid x\in G,x\ne0\}$ of $G^*$  is a starter in $G$. For convenience, we will call a partition of a set of even cardinality into pairs a \textit{2-partition} of this set.

In this paper, we will consider only cyclic additive abelian groups, more precisely, groups $\mathbb{Z}_n$ of integers modulo $n$, where $n\ge3$ is odd.

\begin{df}\label{strong, skew starters}
 A 2-partition  $S=\{\{x_i,y_i\}\}_{i=1}^q$ of $\ZZ^*_{n}$,  $ n=2q+1$, $q\ge 1$ is called

(a) a  starter  in $\ZZ_n$, if
\begin{equation}\label{starter}
\{\pm(x_i-y_i)\pmod{n}|\{x_i,y_i\}\in S,\ 1\le i\le q\}=\ZZ^*_{n};
 \end{equation}

(b)  strong, if
\begin{equation}\label{strong}
\hat S=\{(x_i+y_i)\pmod{n}|\{x_i,y_i\}\in S,\ 1\le i\le q\}\subset \ZZ^*_{n}\quad {\rm and} \quad |\hat S|=q;\end{equation}

(c) skew, if
\begin{equation}\label{skew}
\{\pm(x_i+y_i)\pmod{n}|\{x_i,y_i\}\in S,\ 1\le i\le q\}=\ZZ^*_{n};\end{equation}

(d) cardioidal (\cite{b21}), if all its pairs are cardioidal of order $n$, that is, if each pair of the partition 
\begin{equation}\label{card}\{x,y\}=\{i,2i\pmod{n}\}\end{equation} 
for some $i\in\ZZ_n^*$;

(e) Skolem, if all its pairs are Skolem of order $n$, that is, if each pair $\{x,y\}$ of the partition is such that
\begin{equation} \label{skolem}
y-x \le q\pmod{n} \Leftrightarrow y>x.
\end{equation}  
Here we assume $1<2<...<2q$ to be the order of the non-zero integers modulo $n$.

\end{df} 
We will refer to 2-partitions of $\ZZ_n^*$ as 2-partitions {\it of order} $n$.
A  {\it strong starter} in $\ZZ_n$ is a 2-partition of order $n$ that possesses properties (\ref{starter}) and (\ref{strong}). 
A  {\it skew starter} in $\ZZ_n$ is a 2-partition of order $n$ that possesses properties (\ref{starter}) and (\ref{skew}). 
Clearly, skew 2-partitions comprise a subset of strong 2-partitians. Consequently, any skew starter is strong. Also, it is known that  if a partition is cardioidal, then it is Skolem \cite{b21}. 
All the other implications are absent, and a 2-partition may hold any combination of these properties independently from one another.  

\begin{ex}\label{example of independence of properties}
The 2-partition $R=\{\{1,2\},\{3,4\}\}$  of $\ZZ_5^*$ is strong ($0\ne 1+2\ne 3+4\ne 0\pmod{5}$) and cardioidal $(i=1\, {\rm and}\, 4)$ but not a starter  ($2-1=4-3$) and not skew ($\pm 3=\mp 7 \pmod{5}$).\\
The 2-partition $Q=\{\{1,3\},\{2,5\},\{4,6\},\{7,8\}\}$ of $\ZZ_9^*$ is Skolem and skew: $\{\pm(1+3)\equiv \pm 4\pmod{9},\pm(2+5)\equiv \mp 2\pmod{9}, \pm(4+6)\equiv \pm 1\pmod{9},\pm(7+8)\equiv \mp 3\pmod{9}\}=\ZZ_9^*$. But it is not a starter ($3-1=6-4$), nor is it cardioidal as, for example, the pair $\{1,3\}$ does not satisfy property (\ref{card}).\\
The starter $T=\{\{2,3\},\{4,6\},\{5,1\}\}$ in $\ZZ_7$ is strong and skew: $\{\pm(2+3)\equiv \pm 5\pmod{7}, \pm(4+6)\equiv \pm 3\pmod{7},\pm(5+1)\equiv \mp1\pmod{7}\}=\ZZ_7^*$, but not Skolem, as the pair $\{5,1\}$ does not satisfy property (\ref{skolem}).\\
The starter $S=\{\{1,2\},\{10,12\},\{3,6\},\{4,8\},\{11,16\},\{9,15\},\{7,14\},\{5,13\}\}$ in $\ZZ_{17}$ is strong as all the pairs yeild pairwise different non-zero sums $\pmod{17}$ and Skolem \cite{b09}. However, $S$ is not skew as, for example, the pairs $\{10,12\}$ and $\{4,8\}$ yeild the sums $5\pmod{17}$ and $12\pmod{17}$, respectively, and $5\equiv -12\pmod{17}$. Neither $S$ is cardioidal as, for example, the pair $\{5,13\}$ does not satisfy property (\ref{card}). 
\end{ex}

First, strong starters were introduced by Mullin and Stanton in 1968 \cite{b016} for constructing Room squares  and Howell designs. In 1969, Mullin and Nemeth \cite{b01} gave a general construction for finding these starters in cyclic groups.

The question of the existence (or non-existence) of a strong starter in an abelian group is crucial in the theory of Room squares. We refer readers interested in constructions of strong starters to \cite{b014}, \cite{b05}, \cite{b01} and the references therein.

Strong starters in groups of order 3, 5 and 9 do not exist \cite[p.144]{b06}. It is an open question whether there exists a strong starter in every cyclic group of an odd order exceeding 9. In 1981, Dinitz and Stinson \cite{b07} found (by a computer search) strong starters in the cyclic group of order $n$ for all odd $7\le n\le999,\ n\ne 9$.

At present, the strongest known general statement on the existence of strong starters is the following \cite[p.625]{b014}: \textit{For any $n>5$ coprime to 6, an abelian group of order n admits a strong starter.}

Skew starters give rise to special Room squares called \textit{skew Room squares}, important combinatorial designs. It is known \cite[p.627]{b014} that skew starters of order $n$ do not exist, if $3|n$.
\smallskip

\smallskip
\textit{Skolem starters}, the objects of our close attention, are defined only in $\ZZ_n$. 
 \begin{df}\label{defn of skolem starter1}
Let $n=2q+1$, and $1<2<...<2q$ be the order of the non-zero integers modulo n. A starter in $\mathbb{Z}_n$ is Skolem if it can be written as a set of ordered pairs $\{(s_i,t_i)\}_{i=1}^q$, where $t_i-s_i\equiv i\pmod{n}$ and $t_i>s_i,\ 1\le i\le q$.
\end{df}

Skolem starters received their name \cite{b09} after Skolem sequences. A \textit{Skolem sequence} of order $q$ is a sequence $(s_1,...,s_{2q})$ of integers from $D=\{1,...,q\}$ such that for each $i\in D$ there is exactly one $j\in \{1,...,2q\}$ such that $s_j=s_{j+i}=i$.  Skolem sequences exist iff $q\equiv0$ or $1\pmod{4}$ \cite{b014}. They were originally used by Skolem in 1957 for the construction of Steiner triple systems \cite{b02}. 

Skolem sequences are widely applied in many areas such as triple systems, balanced ternary designs, factorization of complete graphs, starters, labelling graphs. Readers intersted in these applications may address their attention to \cite{b23} and the references therein.

Given a Skolem sequence $(x_1,x_2,...,x_{2q})$, consider all pairs $\{i_k,j_k\}$ such that $j_k>i_k$ and $x_{i_k}=x_{j_k}=k,\ k=1,...,q$. This set of pairs forms a partition of the set $\mathbb{Z}_n^*$ of all non-zero elements of $\ZZ_n$, where $n=2q+1$. Since $j_k-i_k\equiv k\pmod{n}$, (and consequently, $i_k-j_k\equiv -k\pmod{n}),\ k=1,...,q$, this set of pairs is a starter in $\mathbb{Z}_n$. 

\begin{ex}\label{skolem sequence} Sequence $(1,1,5,2,4,2,3,5,4,3)$ is a Skolem sequence of order $5$: the length of the sequence is $2\cdot5=10$, and $x_1=x_2=1,x_4=x_6=2,x_7=x_{10}=3,x_5=x_9=4, x_3=x_8=5$, so $1$'s, $2$'s, $3$'s, $4$'s and $5$'s are one, two, three, four and five positions apart, respectively.  This Skolem sequence yields a starter $T=\{\{1,2\},\{4,6\},\{7,10\},\{5,9\},\{3,8\}\}$ in $\mathbb{Z}_{11}$.
\end{ex}

\begin{lemma}\label{defn of Skolem starter2} A starter $S$ in $\ZZ_k$ is Skolem if and only if all its pairs are Skolem of order $k$. In other words, Skolem starters are partitions with properties (\ref{starter}) and (\ref{skolem}).
\end{lemma}
\begin{proof}
The lemma follows from Definitions \ref{strong, skew starters} and \ref{defn of skolem starter1}.
\end{proof}

Clearly, Skolem starters in $\mathbb{Z}_n$ are in one-to-one correspondence with Skolem sequences of order $q=(n-1)/2$. Therefore, Skolem starters exist in $\mathbb{Z}_n$ iff $n\equiv1$ or $3\pmod{8}$.


{\it Strong Skolem starters} are partitions with properties (\ref{starter}), (\ref{strong}), and  (\ref{skolem}).
The value of strong Skolem starters of order $2q+1$ is in their applicability in constructing Room squares and cubes of order $2q+2$  on one hand, and \textit{Steiner triple systems}, STS($6q+1$), on the other. Recall that an STS($v$) is a collection of $3$-subsets, called $blocks$, of a $v$-set $S$, such that every two elements of $S$ occur together in exactly one of the blocks.
 \begin{theorem}\label{Shalaby theorem} $($Shalaby, $1991\ \cite[pp.60-62]{b09}.)$ For $11\le n\le 57, n\equiv1$ or $3\pmod{8}$, $\mathbb{Z}_n$ admits a strong Skolem starter.
\end{theorem}

\begin{conjecture}\label{Shalaby conjecture} $($Shalaby, $1991\ \cite[p. 62]{b09}.)$  
Every $\mathbb{Z}_n$ with $n\equiv 1$ or $3\pmod{8}$ and $n\ge11,$ admits a strong  Skolem starter.
\end{conjecture}
Up to 2018, there were known only finitely many strong Skolem starters. In 2018, Ogandzhanyants et al. explicitly constructed an infinite family of strong Skolem starters \cite{b21}, proving the following 
\begin{theorem}\label{Ogandzhanyants theorem1}
Let $n=\prod_{i=1}^m p_i^{k_i}$, where $p_i>3,\ i=1,...,m$, are pairwise distinct primes such that $\ord_{p_i}(2)\equiv 2\pmod{4}$, and $k_i\in\NN,\ i=1,...,m$. Then $\ZZ_n$ admits a skew Skolem starter. 
\end{theorem}

In addition, in was shown that all Skolem starters found in \cite{b21} are cardioidal starters, that is, they possess property (\ref{card}), and no strong cardioidal starter lies outside of the family fully described in \cite{b21}. The discovery in \cite{b21} boosted up the attention of some other researchers towards the proof of Conjecture \ref{Shalaby conjecture}, see for example \cite{b26} and the references therein. They explored alternative approaches to constructing strong Skolem starters, but no infinite family of strong Skolem starters other than strong cardioidal starters has been found. 

Theorems \ref{intermediate result} and \ref{main} stated and proved in this paper allow formation of new infinite families of strong (and skew) Skolem starters of composite orders, which are not cardioidal and thus significantly extends the previous result.  

Gross \cite[p.170]{b22} in 1974 indicated a way to produce a starter for the group $G\oplus H$, the direct sum of two finite abelian groups, given a starter for $H$ and a set of starters for $G$. He showed that under certain conditions, strong starters for $H$ and $G$ give rise to a strong starter for $G\oplus H$. Our constructions of the products given in Definitions \ref{WST} and \ref{dfn of cardioidal product} are inspired by that paper. However, in contrast with Gross who focused on the existence of strong starters and starters with adders in a general setting, our constructions are explicitly defined in cyclic groups $\ZZ_n$ as we are concerned with skew and strong Skolem starters. In addition, most of our statements have a converse. 

Our construction of products resembles the one, given by Turgeon in 1979 \cite{b28} for \textit{additive sequences of permutations}, in the general context of difference sets.  Indeed, Skolem starters could be treated as a very special case of difference sets.
However, in this paper we avoid  over-generalization and adapt the presentation specifically to our needs. Thus, we first apply the construction to  2-partitions in $\ZZ_n^*$ 
without any restrictions imposed on them.  Then we endow the 2-partitions with a certain property stated in Definition \ref{strong, skew starters}, apart from all other properties, 
and deduce the direct and inverse relationship of these properties for the resulting construction. Whereas  the direct relations may also be  concluded from the previous research on difference sets, the converse statements are our contribution to the topic.

The structure of this papers is as follows.

In Section 2, we introduce the notion of a product of a pair of  2-partitions of $\ZZ_{2p+1}^*$ and $\ZZ_{2q+1}^*$ respectively.

In Section 3, we focus our attention on starters and other special classes of 2-partitions.  Subsection 3.1 gives some preliminaries. 
In Subsection 3.2, we prove several  properties of the product of two 2-partitions and give an important intermediate result, Theorem \ref{intermediate result}. In Subsection \ref{The parametrized product of starters}, we generalize the initial definition of the product and compare its properties to the initial one. In Subsection \ref{Cardioidal product of two 2-partitions}, we show explicit ways to apply the products of Skolem starters.

In Section 4, we conclude with a discussion of several implications of the statements proved in this paper and give the main result of this paper, Theorem \ref{main}.

\section{The  product of  2-partitions: Construction}\label{section construction}

\begin{lemma}\label{orderx}
 Let $n=2q+1$, $q\ge 1$.  From any 2-partition $S=\{\{a_i,b_i\}\}^q_{i=1}$ of $\ZZ^*_n$, it is possible to make a set of ordered pairs $\bar{S}=\{(x_i,y_i)\}_{i=1}^q$, where either $x_i=a_i, y_i=b_i$ or  $x_i=b_i, y_i=a_i$, 
 for all $1\le i\le q$, such that 
\begin{equation}\label{order}
\cup^q_{i=1}\{\pm x_i\}=\ZZ_n^*. \quad 
\end{equation}
\end{lemma}

\begin{proof}
Let us denote by $\{a,b\}\mapsto (x,y)$ the operation of making  an ordered pair $(x,y)$ from an unordered pair $\{a,b\}$ by setting $x\equiv a,\ y\equiv b\pmod{n}$, then removing $\{a,b\}$ from $S$ and placing $(x,y)$ in $\bar S$.  Below we describe an explicit  construction of $\bar S$. 

First,  for each pair $\{a_i,b_i\}$, where $b_i\ \equiv -a_i\pmod{n}$, if exists,  set $\{a_i,-a_i\}\mapsto (x_i,y_i)$, and place it in the end of the list $\bar S$.  From all other pairs remaining  in $S$,
pick any pair, say, $\{a_1,b_1\}\in S$, and set $\{a_1,b_1\}\mapsto (x_1,y_1)$.  Then, find a pair which the element $-b_1$ belongs to. Without loss of generality (WLOG),  let $a_k\equiv-b_1\pmod{n}$.  Then set $\{-b_1,b_k\}\mapsto (x_2,y_2)$. Then, find a pair which the element $-b_k$ belongs to. WLOG,  let $a_j\equiv-b_k\pmod{n}$.  Then set $\{-b_k,b_j\}\mapsto (x_3,y_3)$.
 And so on, until the element $-a_1$ appears in some pair $\{a_m,b_m\}\in S$, which will produce  $(x_l,y_l)$, where $y_l\equiv -a_1\pmod{n}$. 

 Note, that by the construction  $x_2\equiv -y_1\pmod{n}$, $x_3\equiv -y_2\pmod{n}$, etc.
Clearly, such a collection of pairs spans over the subset  $\{x_1,y_1,-y_1,y_2,-y_2,...,-x_1\}\subset\ZZ_n^*$. 

Finally, pick any remaining pair in $S$, give it an order and continue the process until all the pairs are ordered.

\end{proof}

\begin{remark}
Note that for the set $\bar S$ described  in Lemma \ref{orderx}, we automaticaly have:
\begin{equation}\label{ordery}
\cup^q_{i=1}\{\pm y_i\}=\ZZ_n^*. \quad 
\end{equation}

\end{remark}

\begin{ex}
For $q=6$ and partition $\{ \{1,12\}; \{2,3\}; \{4,6\}; \{5,7\}; \{8,9\}; \{10,11\}\} $ we form $\bar S$ by making  the following 3 clusters:
$$
 \quad \bar S=\{ (2,3), (-3,-2), \quad (4,6), (-6,5), (-5,-4),\quad (1,-1)\pmod{13}\}
$$
with property  (\ref{order}), as required:
$$
\{\pm 2, \mp 3, \pm 4, \mp 6, \mp 5, \pm 1\pmod{13}\}= \ZZ_{13}^*.
$$
\end{ex}

Let  $S=\{\{a_i,b_i\}\}^q_{i=1}$  be a 2-partition of $\ZZ_{2q+1}^*$.
Denote by $\bar S$ a set  of $q$  {ordered} pairs of $S$, that obey property (\ref{order}):
\begin{equation}\label{SS}
 \bar S=\{(x_i,y_i)\}_{i=1}^q.
\end{equation}
The existence of these sets is sequred  by Lemma \ref{orderx}.
In addition, by $\bar S'$ we denote the set $\bar S'=\{(-x_i,-y_i)\}_{i=1}^q$ and
by $\tilde S$ we denote a set of arbitrary ordered pairs of $S$. 

\begin{df}\label{WST}
Given two 2-partitions, $S$ of $\ZZ_{2q+1}^*$ and $T$ of $\ZZ_{2p+1}^*$, let us form  sets of ordered pairs $\tilde{S}$, and $\bar{T},\ \bar{T'}$ as specified above. Consider the set  $W_{ST}=\{\{u_i,v_i\}\}^{k}_{i=1}$ of $k=2qp+ q+p$ pairs  of the form
\begin{equation}\label{uv}
\{u,v\},\quad  {\rm where}\quad u=(2q+1)r+x,\quad v=(2q+1)t+y
\end{equation}
 divided in the following  types:\\
(i) $(2p+1)q$ pairs: one for each  $(r,t)\in \bar{T}\cup \bar{T'}\cup \{(0,0)\}$ and for each $(x,y)\in \tilde S$ and\\
(ii)  $p$ pairs: one for each $(r,t)\in \bar{T}$ and $x=y=0$.\\

We will call the set of pairs $W_{ST}$ a product of $S$ and $T$.
\end{df}

\begin{ex}\label{ex1}

Let us construct the set $W_{ST}$ for 2-partitions $S=T=\{1,2\}$  of $\ZZ_3^*$.  Here 
$q=p=1$. 
Take $\tilde S= \bar T=\{(1,2)\}$ and $\bar T'=\{(2,1)\}$. 
So we have the pairs of the two types:\\
(i) $\{3\times0+1,3\times0+2\}=\{1,2\}$, $\{3\times 1+1,3\times 2+2\}=\{4,8\}$,
 $\{3\times 2+1,3\times 1+2\}=\{7,5\}$\\
(ii) $\{3\times 1+0,3\times 2+0\}=\{3,6\}$\\
The set of these four pairs constitutes $W_{ST}$.  

\end{ex}

\begin{remark}\label{starred types}
The pairs of $W_{ST}$ can be formed in various ways, depending on the choices  made in the process of constructing  $\tilde S$ and $\bar T$ from the 2-partitions $S$ and $T$. So, a product of the two 2-partitions is not unique. Nevertheless, the properties proven below hold for $W_{ST}$ regardless of the ordering choices. 

We will consider an alternative construction of a product in Sections 3.3 and 3.4. As well, we will outline more possibilities in Section \ref{Conclusion}.

\end{remark}

A product $W_{ST}$ of two 2-partitions, $S$ of $\ZZ_n^*$ and $T$ of $\ZZ_m^*$, as we will prove, preserves some properties of the factors. The most general one is given in Theorem \ref{W is a partition} below. Prior to that we recall the following simple results needed in the proofs. 

\begin{lemma}\label{modul}
	Let $m,n$ be any natural numbers and $a,b,c,d$ be integers. 
	\begin{enumerate}
		\item If $a\equiv b \pmod{mn}$ then $a\equiv b \pmod{n}$ and $a\equiv b \pmod{m}$.
		
		\item If $(an+c)\equiv (bn+c) \pmod{mn}$ then $a\equiv b \pmod{m}$.
		
		\item  If $(an+c)\equiv (bn+d) \pmod{n}$ then $c\equiv d \pmod{n}$.
		
		\item Let $X_c^n$ be a finite multiset of integers congruent modulo ${n}$ to a given integer $c$, not necessarily all distinct. If $|X_c^n|>m$ then there exist $b, d \in X_c^n$ such that $b=d \pmod {mn}$. 
		
	\end{enumerate}
\end{lemma}
\begin{proof} 
	1. By definition, $a\equiv b \pmod{mn}$ means $(mn) | (a-b)$. But then $m | (a-b)$, so $a\equiv b \pmod{m}$, and $n | (a-b)$, so $a\equiv b \pmod{n}$.
	
	2.  Similarly, $(an+c)\equiv (bn+c) \pmod{mn}$ means $(mn)| ((a-b)n)$. Then $m | (a-b)$, so $a\equiv b \pmod{m}$.
	
	3. As well, $(an+c)\equiv (bn+d) \pmod{n}$ means $n|((a-b)n+(c-d))$. Then $n | (c-d)$, so $c\equiv d \pmod{n}$.
	
	4. For any integer $0\le c< mn$ there are exactly $m$ numbers $0\le a < mn$ congruent modulo $n$ to $c$. Thus, for any multiset $X_c^n$  of more that $m$ integers, the Dirichlet principle implies the existence of  $b=d \pmod {mn}$.
\end{proof}

\begin{theorem}\label{W is a partition} Let $n,m\ge 3$ be odd integers and $S$ and $T$ be 2-partitions of $\ZZ_n^*$ and $\ZZ_m^*$ respectively. Their product $W_{ST}$ (Definition \ref{WST}) is a 2-partition of $\ZZ_{mn}^*$.
\end{theorem}
\begin{proof} 

Let $n=2p+1$ and $m=2q+1$, $p,q\ge 1$.  Let us also establish the natural order in $\ZZ_k:\ 0<1<...<k-1,\ k\in \NN$. 
By definition,  $W_{ST}$ consists of  $2pq+p+q$ pairs, totaling to $4pq+2p+2q=mn-1$ elements, which equals the cardinality of  $\ZZ_{mn}^*$. It remains to show that all these elements of $W_{ST}$ are distinct modulo $mn$. Indeed, all the elements of the pairs of type (ii) are distinct as $T$ is a 2-partition of $\ZZ_m^*$, and they are multiples of $n$. 
All the elements of the pairs for $r=t=0$ and $(x,y)\in \tilde S$ are distinct and  less than $n$ because $S$ is a 2-partition of $\ZZ_n^*$.  All the remaining elements of the pairs of type (i) are greater than $n$ and are not multiples of $n$. Assume for the sake of contradiction the possibility that among them there is a pair $\{u_1,v_1\}$ and a pair $\{u_2,v_2\}$ with a non-empty intersection. Here 
\begin{equation}\label{u1v1}
u_1=r_1n+x_1,\, v_1=t_1n+y_1 \quad u_2=r_2n+x_2,\, v_2= t_2n+y_2,
\end{equation}
and  WLOG,  we let $(x_i,y_i)\in\tilde{S},\ i=1,2,\ (r_1,t_1)\in\bar{T},\ (r_2,t_2)\in\bar{T'}$.

Let for example, $u_1\equiv u_2\pmod{mn}$, that is $(r_1n+x_1)\equiv (r_2n+x_2)\pmod{mn}$. Then, by Lemma \ref{modul}(1\&3), $x_1\equiv x_2\pmod{n}$. Therefore, by Lemma \ref{modul}(2), $r_1\equiv r_2\pmod{m}$. But this is impossible due to property (\ref{order}). Similar argument will lead to a contradiction if one assumes  $v_1\equiv v_2\pmod{mn}$ or $u_1\equiv v_2\pmod{mn}$ or $u_2\equiv v_1\pmod{mn}$.

 This proves that all $4pq+2p+2q=mn-1$ elements appeared in the pairs  of $W_{ST}$ are distinct. Therefore $W_{ST}$ is a 2-partition of $\ZZ^*_{mn}$.
\end{proof}

\section{The product of special classes of 2-partitions}
\subsection{Preliminaries}
The 2-partitions of $\ZZ_n^*$ we mainly concern with are strong and skew Skolem starters. Lemma \ref{orderx} applies to starters in $\ZZ_{2q+1} $ as they form a  2-partition of $\ZZ^*_{2q+1}.$

Before we get to the properties of the product of two starters, we present some additional definitions and a lemma which will be helpful in the sequel.
\begin{df}\label{defn of pairs}
 A pair $\{x,y\}\in S$ is called a canonical pair of order $k$ if $\{x,y\}=\{i,-i\pmod{k}\}$ for some $i\in\ZZ_k^*$. If all pairs of $S$ are canonical, then $S$ is called a canonical starter of order $k$.

\end{df}
 

\begin{df} Two 2-partitions S and S' in the same group are called conjugate if $\{x,\ y\}\in S$ implies $\{-x,\ -y\}\in S'$.
\end{df}

Obviously, every 2-partition has a conjugate. Note that the 2-partition of $\ZZ_n^*$ which is a canonical starter is always conjugate to itself. Moreover, a starter is canonical if and only if it is \textit{self-conjugate}.



The  following properties of conjugate 2-partitions are rather trivial as each of them follows immediately from the definitions of their counterparts, but very important:

\begin{lemma}
 If a 2-partition is either a starter, or canonical, or strong, or skew, or Skolem, or cardioidal, so is its conjugate.
\end{lemma}

\subsection{Properties of the product of two 2-patitions}\label{the properties of two 2-patitions} 

 In Example \ref{ex1}, the two 2-partitions we use are starters in $\ZZ_3$. (We have no choice as the only 2-partition of $\ZZ_3^*$ is a starter in $\ZZ_3$). And their product turns out to be a starter in $\ZZ_{3\cdot 3}=\ZZ_9$.

Consider the product of  two starters from different groups.

\begin{ex} \label{ex2}

Let us construct the set $W_{ST}$ for starters $S=\{\{1,4\},\{2,3\}\}$  in $\ZZ_5$ and $T=\{1,2\}$  in $\ZZ_3$.  In this case $n=5, m=3$ and $q=2, p=1$. 

Take $\tilde S=\{(1,4), (2,3)\}$, $\bar T=\{(1,2)\}$, $\bar T'=\{(2,1)\}$.

Then we have the pairs of the two types:\\
(i) $\{5\times0+1,5\times0+4\}=\{1,4\},\ \{5\times0+2,5\times0+3\}=\{2,3\}$

 $\{5\times 1+1,5\times 2+4\}= \{  6, 14\}   $, $\{5\times 1+2,5\times 2+3\}= \{7,13\}   $  

$\{5\times 2+1,5\times 1+4\}= \{  11, 9\}   $, $\{5\times 2+2,5\times 1+3\}= \{12,8\}   $\\
(ii) $\{5\times 1+0, 5\times 2+0\}=\{5,10\}$\\
The set of these seven pairs constitutes $W_{ST}$. In fact, $W_{ST}$ is a starter in $\ZZ_{3\cdot 5}=\ZZ_{15}$.

\end{ex}
 This is not coincidental. A product of two starters is a starter. It turns out that the converse is true as well, that is, if $W_{ST}$ is a starter, then both $S$ and $T$ are starters. The following theorem secures this property.

\begin{theorem}\label{W is a starter} Let $n,m\ge 3$ be odd integers and $S$ and $T$ be 2-partitions of $\ZZ_n^*$ and $\ZZ_m^*$ respectively. Their product $W_{ST}$ (Definition \ref{WST}) is a starter in $\ZZ_{mn}$ if and only if $S$ is a starter in $\ZZ_n$ and $T$ is a starter in $\ZZ_m$.
\end{theorem}
\begin{proof} 

(a) Sufficiency.

Let $S$ be a starter in $\ZZ_n$ and $T$ be a starter in $\ZZ_m$. In order to prove that $W_{ST}$ is a starter in $\ZZ_{mn}$,
we need to show that $W_{ST}$ is a partition of $\ZZ^*_{mn}$ into  pairs $\{\{u_i,v_i\}\}_{i=1}^{(mn-1)/2}$ such that
\begin{equation}\label{starter's differences}
\{\pm(u_i-v_i)\pmod{mn}|\{u_i,v_i\}\in W_{ST}\}=\ZZ^*_{mn}.
\end{equation}
Now, let us look at the differences $\pm( u_k-v_k)\pmod{mn}$, $1\le k\le \frac{mn-1}2$.

Since $T$ is a starter in $\ZZ_m$, the pairs of type (ii) make all possible $m-1$ differences of the form $n\Delta$, where $\Delta\in \ZZ_m^*$.

Consider two distinct pairs $\{u_1,v_1\}$ and $ \{u_2,v_2\}$ of type (i). Suppose, for the sake of contradiction, that $(u_1-v_1)\equiv (u_2-v_2) \pmod{mn}$.  

 Using notation (\ref{u1v1}), we have 
\begin{equation}\label{W starter1}
[(r_1 n+x_1)-(t_1n+y_1)] \equiv [(r_2n+x_2)- (t_2n+y_2)]\pmod{mn}.
\end{equation}

By Lemma \ref{modul}(1), equation (\ref{W starter1}) implies $$[(r_1 n+x_1)-(t_1n+y_1)] \equiv  [(r_2n+x_2)- (t_2n+y_2)]\pmod{n}.$$
Then, by Lemma \ref{modul}(3), we obtain $$(x_1-y_1) \equiv (x_2- y_2)\pmod{n}.$$
Since $S$ is a starter in $\ZZ_n$, it is possible  if and only if $\{x_1,y_1\}= \{x_2,y_2\}$.  
WLOG, assume that this pair is ordered by $(x_1,y_1)= (x_2,y_2)=(x,y)\in \bar{S}$. We have
\begin{equation}\label{W starter2}
((r_1 n+x)-(t_1n+y)) \equiv  ((r_2n+x)- (t_2n+y))\pmod{mn}.
\end{equation}

By Lemma \ref{modul}(2), equation (\ref{W starter2}) implies $(r_1-t_1)\equiv (r_2- t_2)\pmod{m}$. Since $T$ is a starter in $\ZZ_m$, it is possible  if and only if  $r_1=r_2$ and $t_1=t_2$, which contradicts our assumption that  $\{u_1,v_1\}$ and $ \{u_2,v_2\}$ are  distinct pairs. 

\smallskip

(b) Necessity.

Suppose that at least one of the 2-partitions $S$ and $T$ is not a starter of the corresponding group. Then to show that $W_{ST}$ is not a starter, it suffices to find at least two pairs of $W_{ST}$ which produce the same differences.

If $T$ is not a starter then it contains at least two pairs $\{r_1,t_1\},\ \{r_2,t_2\}$ such that $\{\pm(r_1-t_1)\}=\{\pm(r_2-t_2)\} \pmod{m}$. Consequently, two pairs $\{r_1n,t_1n\},\ \{r_2n,t_2n\}$ in $W_{ST}$ of type (ii) will yield the same differences modulo $mn$.

If $S$ is not a starter then it contains at least two pairs $\{x_1,y_1\},\ \{x_2,y_2\}$ such that $\{\pm(x_1-y_1)\}=\{\pm(x_2-y_2)\}\pmod{n}$. Then there are $2m$ pairs in $W_{ST}$ of types (i), which produce differences congruent to $\pm (x_1-y_1)$ modulo $n$. They are $\{x_1,y_1\},\ \{x_2,y_2\}$, $\{r_i n+x_1,t_i n+y_1\},\ \{r_in+x_2,t_in+y_2\}$, $\{-r_i n+x_1,-t_i n+y_1\},\ \{-r_in+x_2,-t_in+y_2\}$, $1\le i\le p$.

Hence, by Lemma \ref{modul} (4), we conclude that there are two pairs among these $2m$ pairs that satisfy $\{\pm(u_1-v_1)\}=\{\pm(u_2-v_2)\}\pmod{mn}$. So,
(\ref{starter's differences}) is impossible, which means $W_{ST}$ is not a starter in $\ZZ_{mn}$.
This completes the proof of necessity. 

\end{proof}

Next statement clarifies the conditions for obtaining a strong 2-partition.

\begin{theorem}\label{W is strong} Let $n,m\ge 3$ be odd integers and $S$ and $T$ be 2-partitions of $\ZZ_n^*$ and $\ZZ_m^*$ respectively. Then their product $W_{ST}$ (Definition \ref{WST}) is a strong 2-partition of $\ZZ_{mn}^*$ if and only if $S$ is strong and $T$ is skew. 
\end{theorem}
\begin{proof} 

(a) Sufficiency.

Let $S$ be strong and $T$ be skew. To show that $W_{ST}$ is strong, we have to show that if $\{u_1,v_1\}$ and $\{u_2,v_2\}$ are two distinct pairs in $W_{ST}$ then $u_1+v_1\not\equiv u_2+v_2\pmod{mn}$, and for any $\{u,v\}\in W_{ST}$ there holds $u+v\not\equiv 0\pmod{mn}$.

Suppose, for the sake of contradiction,
\begin{equation}\label{thm2.1}
u_1+v_1\equiv u_2+v_2\pmod{mn}.
\end{equation}

Using notation (\ref{u1v1}), we have $$(r_1+t_1)n+x_1+y_1\equiv (r_2+t_2)n+x_2+y_2  \pmod{mn}.$$

Consequently, by Lemma \ref{modul}(1), we obtain $$(r_1+t_1)n+x_1+y_1\equiv (r_2+t_2)n+x_2+y_2  \pmod{n}.$$

Then, by Lemma \ref{modul}(3), $x_1+y_1\equiv x_2+y_2\equiv C \pmod{n}$. If $C=0$ we get pairs of type (ii). Otherwise, 
since $\{x_i,y_i\}\in S,\ i=1,2$, and $S$ is strong, we conclude  $\{x_1,y_1\}=\{x_2,y_2\}$. In either case, by Lemma \ref{modul}(2), (\ref{thm2.1}) implies $r_1+t_1\equiv r_2+t_2\pmod{m}$.  Here, the pairs $(r_1,t_1), (r_2,t_2)$ are from either set $\bar T$ or $\bar T'$. By the hypothesis of the theorem, $T$ is skew, which means that all sums of the pairs of $T$ along with $T'$ are different $\pmod{m}$. 
Thus, there are two options:
\begin{enumerate}
\item either $r_1=r_2$ and $t_1=t_2$,
\item or $r_1=t_2$ and $t_1=r_2$.
\end{enumerate}
Case 1 contradicts our assumption that $\{u_1,v_1\}$ and $\{u_2,v_2\}$ are two distinct pairs in $W_{ST}$.
Case 2 is impossible due to the following reason. Let $r_1=t_2=r$ and $t_1=r_2=t$. WLOG, assume  $(r,t)\in \bar{T}$ and $(t,r)\in \bar{T'}$. But $(-r,-t)\in \bar{T'}$, and $r-t=(-t)-(-r)$, which implies that $t\equiv -r\pmod{m}$. The latter means that $\{-r,r\}\in T$, which  contradicts our assumption that $T$ is skew ($T$ is not even strong in that case since $-r+r\equiv 0\pmod{m}$). 

Finally, let $\{u,v\}\in W_{ST}$, $u=rn+x, v=tn+y$. If $u+v\equiv 0\pmod{mn}$, then by Lemma \ref{modul}(1 \& 3), $x+y\equiv 0\pmod{n}$, which is impossible since $S$ is strong. 

This completes the proof that $W_{ST}$ is a strong 2-partition $\ZZ_{mn}$.

\smallskip

(b) Necessity.

If $S$ is not strong, then, regardless of the properties of $T$, there are two possible cases:
\begin{enumerate}
\item $S$ contains a pair $\{x,y\}$ such that $x+y\equiv 0\pmod{n}$. Then consider all the pairs of the type (i) and of the form $\{rn+x,tn+y\}$. There are exactly $m$ such pairs. These $m$ pairs along with $(m-1)/2$ pairs of type (ii) yield $(3m-1)/2$ sums in $\ZZ_{mn}$, which are congruent to 0 modulo $n$. But these sums can not be all different and non-zero modulo $mn$ by Lemma \ref{modul}(4). Thus $W_{ST}$ is not strong.

\item $S$ contains two pairs $\{x_1,y_1\}$ and $\{x_2,y_2\}$ such that $x_1+y_1\equiv x_2+y_2\equiv c\pmod{n}$. Then consider all the pairs of the type (i) and of the form $\{rn+x_i,tn+y_i\},\ i=1,2$. There are $2m$ of them, and all of them yield sums in $\ZZ_{mn}$, which are congruent to $c$ modulo $n$. By Lemma \ref{modul} (4), the sums can not be all different modulo $mn$.  Thus $W_{ST}$ is not strong.

\end{enumerate}

If $T$ is not skew, some of the pairs, say, $(r_1,t_1)\in\bar{T}$ and $(r_2,t_2)\in\bar{T'}$ yield the same sum $\pmod{m}$. Let us take a pair $(x,y)\in\tilde{S}$. Then pairs
$\{r_1n+x,t_1n+y\}$ and $\{r_2n+x,t_2n+y\}$ produce the same sum modulo $mn$. Hence $W_{ST}$ is not strong.

\end{proof}
\begin{remark}
Not every 2-partition of a composite order is a product of two 2-partitions. For example, there are known \cite{b05} strong starters of order $3p$ for some prime $p>3$, but obtaining them by means of a product of two starters of orders 3 and $p$ respectively would have contradicted Theorem \ref{W is strong}.
\end{remark}

The following Theorem clarifies the question whether or not $W_{ST}$ is skew.

\begin{theorem}\label{skew theorem} Let $n,m\ge 3$ be odd integers and $S$ and $T$ be 2-partitions of $\ZZ_n$ and $\ZZ_m$ respectively.  Then their product $W_{ST}$ (Definition \ref{WST})  is skew if and only if both $S$ and $T$ are skew. 
\end{theorem}
\begin{proof}

 (a) Sufficiency.

 Let $S$ and $T$ be skew and let $\{u_1,v_1\}$ and $\{u_2,v_2\}$ be two arbitrary distinct pairs of $W_{ST}$. By Theorem \ref{W is strong}, we know that $W_{ST}$ is strong, that is, $u_1+v_1 \not\equiv u_2+v_2 \pmod{mn}$. To show that $W_{ST}$ is skew, it remains to show that there holds
\begin{equation}\label{skewness condition}
u_1+v_1 \not\equiv - (u_2+v_2) \pmod{mn}.
\end{equation}

 Suppose, for the sake of contradiction, that (\ref{skewness condition}) is not true, that is, in notation (\ref{u1v1}),
 $$
 r_1n+x_1+t_1n+y_1\equiv - (r_2n+x_2+t_2n+y_2) \pmod{mn}.
 $$
 By Lemma \ref{modul}(1 \& 3), we obtain $x_1+y_1 \equiv -(x_2+y_2) \pmod{n}$, which is impossible as  $S$ is skew, unless $\{\{u_i,v_i\}\}_{i=1}^2$ are pairs of type (ii). But then $(r_1n+t_1n)\equiv -(r_2n+t_2n)\pmod{mn}$, and hence, by Lemma \ref{modul}(1) $(r_1+t_1)\equiv -(r_2+t_2)\pmod{m}$, which is impossible, since $T$ is skew.
 
This contradiction implies that for any two distinct pairs in $W_{ST}$, $\{u_1,v_1\}$ and $\{u_2,v_2\}$, there holds (\ref{skewness condition}). Therefore, $W_{ST}$ is skew.

\smallskip

(b) Necessity.

By Theorem \ref{W is strong}, if $T$ is not skew, then $W_{ST}$ is not skew.

Now, suppose, $S$ is strong but not skew, then $\bar{S}$ contains at least two pairs $(x_1,y_1)$ and  $(x_2,y_2)$ such that $(x_1+y_1)\equiv -(x_2+y_2)\equiv c\pmod{n}$. It is clear that $c\not\equiv 0\pmod{n}$, as $S$ is strong. 

There are  $m$ pairs $\{rn+x_1,tn+y_1\}$ of the type (i) in the starter $W_{ST}$.There are also  $m$ pairs $\{rn+x_2,tn+y_2\}$ of the type (i) in the conjugate 2-partition $W'_{ST}$.

These pairs yield $2m$ sums in $\ZZ_{mn}$, which are congruent to $c$ modulo $n$. But they could not be all different modulo $mn$ by Lemma \ref{modul} (4). We conclude that there are two pairs among these $2m$ pairs that satisfy $\{\pm(u_1+v_1)\}=\{\pm(u_2+v_2)\}\pmod{mn}$. Hence, $W_{ST}$ is not skew.

\end{proof}

Finally, we deal with Skolem 2-partitions. 

\begin{theorem}\label{Skolemness} Let $n,m\ge 3$ be odd integers and $S$ and $T$ be 2-partitions of $\ZZ_n$ and $\ZZ_m$ respectively.  Then their product
$W_{ST}$ (Definition \ref{WST}) is a Skolem 2-partition of $\ZZ_{mn}$ if and only if $S$ and $T$ are both Skolem 2-partitions of $\ZZ_{n}$ and $\ZZ_{m}$ respectively.
\end{theorem}
\begin{proof}
 Let us order $\ZZ_{k}: 0<1<...<k-1,\ k\in\NN$.

(a) Sufficiency.

Let $S$ and $T$ be Skolem 2-partitions of orders $n$ and $m$ respectively. To show that $W_{ST}$ is Skolem, we have to show that all its pairs $\{u,v\}$ are Skolem pairs of order $mn$, that is, 
$u<v$ and $v-u\le \frac{mn-1}2$.

Let  the pair $\{x^*,y^*\}\in S$ make the greatest difference in $S$, that is, $x^*<y^*,\ y^*-x^*\le\frac{n-1}2$. As well, consider the pairs in $\{r^*,t^*\}\in T$ and $\{-t^*,-r^*\}\in T'$, which make the greatest difference $t^*-r^*=(-r^*)-(-t^*)\le \frac{m-1}2$.

Each of the pairs $\{u,v\}$ in the form either $(x_i,y_i)$ or $(r_jn,t_jn)$  is Skolem because for any $n\ge 1$ and $m\ge 1$ we have:
$$
 v-u\le y^*-x^*\le\frac{n-1}{2}<\frac{mn-1}{2}, \qquad v-u\le (t^*-r^*)n\le\left(\frac{m-1}{2}\right)n=\frac{mn-n}{2}<\frac{mn-1}{2}.
$$
 
For other pairs of type (i) we consider two cases:
\begin{enumerate}
\item $(x^*,y^*)\in\bar{S}$. Then a pair $\{u,v\}\in W_{ST}$, where $u=r^*n+x^*<\ v=t^*n+y^*$, makes the greatest possible difference among the pairs of $W_{ST}$,
$$v-u=t^*n+y^*-(r^*n+x^*)\le\frac{m-1}{2}n-\frac{n-1}{2}=\frac{mn-1}{2}.$$
\item  $(y^*,x^*)\in\bar{S}$. Then a pair $\{u,v\}\in W_{ST}$, where $u=-r^*n+y^*>\ v=-t^*n+x^*$, makes the greatest possible difference among the pairs of $W_{ST}$,
$$u-v=-r^*n+y^*-(-t^*n+x^*)\le\frac{mn-1}{2}.$$

\end{enumerate}

All other pairs of type (i) are clearly Skolem as they make no difference greater than $\frac{mn-1}{2}$.

(b) Necessity.

If $T$ is not Skolem, then it contains a pair $\{r,t\}$ which is not Skolem, that is, given $r<t$, we have $t-r\ge\frac{m+1}{2}$. Then the corresponding pair of type (ii), $\{rn,tn\}\in W_{ST}$, yields a difference greater than $\frac{mn-1}{2}$:
$$tn-rn\ge\frac{m+1}{2}n=\frac{mn+n}{2}>\frac{mn-1}{2}.$$

If $S$ is not Skolem, there is $\{x,y\}\in S$ which is not Skolem, that is, given $x<y,\ y-x\ge\frac{n+1}{2}$. Now, let us take the pair $\{r,t\}\in T$, such that $r<t$ and $t-r\ge\frac{m-1}{2}$. WLOG, assume $(x,y)\in\bar{S}$ and $(r,t)\in \bar{T}$. Then the pair $\{u,v\}\in W_{ST},\ u=rn+x<v=tn+y$ makes the difference
$$v-u \ge\frac{m-1}{2}n+\frac{n+1}{2}=\frac{mn-n+n+1}{2}=\frac{mn+1}{2}.$$
That means $\{u,v\}\in W_{ST}$ is not a Skolem pair (\ref{skolem}) of order $mn$ and hence, by Lemma \ref{defn of Skolem starter2}, $W_{ST}$ is not Skolem.
\end{proof}

 Let us summarize the results of this subsection.

\begin{theorem}\label{intermediate result}
 Let $S$ and $T$ be 2-partitions in $\ZZ_n$ and $\ZZ_m$ respectively.

1. If both $S$ and $T$ are Skolem starters and, in addition, $S$ is strong and $T$ is skew, then the product $ W_{ST}$ (Definition \ref{WST}) is a strong Skolem starter in $\ZZ_{nm}$. Moreover, if $S$ and $T$ are both skew and Skolem in their groups, then $ W_{ST}$ is a skew Skolem starter in $\ZZ_{nm}$.

2. If the product $W_{ST}$ of  partitions $S$ and $T$ is a strong but not skew Skolem starter then $S$ is a strong but not skew Skolem starter and $T$ is a skew Skolem starter. If the product $W_{ST}$ of  partitions $S$ and $T$ is a skew Skolem starter then  both $S$ and $T$ are skew Skolem starters.

\end{theorem}
\begin{proof}The statement follows from Theorems \ref{W is a starter}, \ref{W is strong}, \ref{Skolemness}, \ref{skew theorem}.
\end{proof}
\begin{remark}The direct part of Theorem \ref{intermediate result} can be obtained by combining the idea and constructions of Gross in \cite{b22}, Turgeon in \cite{b28} and Chen et al \cite{b27}. The converse statement requires a more general consideration of the product of two partitions.
\end{remark} 
\begin{ex}\label{order 187}
 
Let us choose  the strong (but not skew) Skolem starter $S$ of order 17 from Example \ref{example of independence of properties}. Using Definition \ref{WST}, it is possible to generate strong Skolem starters of orders $17m$, where $m$ is one of the orders of the known skew Skolem starters. By Theorem \ref{Ogandzhanyants theorem1}, there are infinitely many such starters. Paper \cite{b21} gives an explicit way of constructing a family of cardioidal starters (\ref{card}). It was proven in \cite{b21} that every cardioidal starter is skew unless its order is divisible by 3. 
Take for example the following cardioidal starter of order $11$: $T=\{\{1,2\},\{7,9\},\{3,6\},\{4,8\},\{5,10\}\}$. Since $T$ is a skew Skolem starter, $W_{ST}$ is a strong Skolem starter in  $\ZZ_{17\cdot 11}=\ZZ_{187}$.

\end{ex}
	
	Definition \ref{WST} can be modified in a variety of ways to achieve diversity of the obtained 2-partitions. For example, we could make  an alternative construction of the pairs (\ref{uv}) of $W_{ST}$:

(i$^*)\ 2pq$ pairs: one for each  $(x,y)\in \bar{S}\cup \bar{S'}$ and for each $(r,t)\in \tilde T$ and $q$ pairs for $r=t=0$ and $(x,y)\in \bar{S}$. 

(ii$^*)\ p$ pairs: one for each $(r,t)\in \tilde T$ and $x=y=0$.\\
The proofs of the statements, involving this way of constructing $W_{ST}$, will be analogous. 

This modification, while significantly expanding the variety of the obtained starters, does not let us produce a strong starter as a product of two strong starters. For example, we cannot obtain a strong Skolem starter of order $17^2=289$ out of the starter $S$ of order 17 used in Example \ref{order 187}. To achieve this objective, we have to further modify Definition \ref{WST}. 

\subsection{The product of 2-partitions with a nucleus}\label{The parametrized product of starters}

In this section, we construct a  family of products by introducing the following object.
	
\begin{df}\label{X}  The set of ordered pairs $X_m=\{(u_i,v_i)\}_{i=1}^{m-1}\subset \ZZ_m^*\times\ZZ_m^*$  is called a nucleus of order $m$ if $\cup_{i=1}^{m-1} u_i=\cup_{i=1}^{m-1} v_i=\ZZ_m^*$.

\end{df}

Then we define a product  with nucleus $X_m$.

\begin{df}\label{WXST} Let $S$ and $T$ be 2-partitions  of $\ZZ_n^*$ and $\ZZ_m^*$ respectively, $ n=2q+1,\ m=2p+1, \, q,p\ge 1$. Let $X_m$ be a nucleus of order $m$.
The set $W^X_{ST}$ consisting of pairs of the form(\ref{uv}) divided in the following  types:

	 (i$_X)\ mq$ pairs: one for each $(x,y)\in \tilde S$ and $(r,t)\in \{0,0\} \cup X_m$;
	
	 (ii$_X)\ p$ pairs: one for each $(r,t)\in \tilde T$ and $x=y=0$,

 is called the $X$-generated product of $S$ and $T$.
\end{df}
	Below we will show that different choices of nucleus $X$ will lead to different properties of $W_{ST}^X$. It turns out that the use of a proper nucleus allows us to loosen the hypotheses of Theorems \ref{W is strong} and \ref{intermediate result}. Consequently, we can extend the family of strong Skolem starters and extend the list of orders $n$ such that $\ZZ_n$ admits a strong Skolem starter.

\begin{theorem}\label{WX is a partition}The set $W^X_{ST}$ (Definitions \ref{X} and \ref{WXST}) is a partition of $\ZZ_{mn}$.
\end{theorem} 
\begin{proof}  The proof is analogous to that of  Theorem \ref{W is a partition} as the set of ordered pairs $X$ has all the properties of $\bar T\cup \bar T'$ used to prove that Theorem. 
(Similar proof of this statement was offered in \cite{b25}).
\end{proof}
\begin{df}\label{Xproperties} A nucleus $X_m$ of order $m$ is called
\begin{enumerate}
\item  subtractive in $\ZZ_m$, if $\{(u_i-v_i)\pmod{m}\}_{i=1}^{m-1}=\ZZ_m^*$;
\item skew in $\ZZ_m$, if $\{(u_i+v_i)\pmod{m}\}_{i=1}^{m-1}=\ZZ_{m}^*$;
\item Skolem in $\ZZ_m$, if it consists of Skolem pairs (\ref{skolem}) of order $m$.
\end{enumerate}
\end{df}

\begin{theorem}\label{WX is a starter} 
The set $W^X_{ST}$ (Definitions \ref{X} and \ref{WXST}) is a starter in $\ZZ_{mn}$ if and only if $X_m$ is subtractive in $\ZZ_m$, and $S$ and $T$ are starters in $\ZZ_n$ and $\ZZ_m$ respectively.
\end{theorem}
\begin{proof}  
(a) Sufficiency

The proof of sufficiency is analogous to that of  Theorem \ref{W is a starter} as the set of ordered pairs $X$ has all the properties of $\bar T\cup \bar T'$ used to prove the part (a) of that Theorem. 

(b) Necessity

Since it is subtractive, one can rewrite the set $X_m$  in the form $X_m=\cup_{i=1}^{(m-1)/2} P_i$, where 
$P_i=\{(u'_i,v'_i), (u''_i,v''_i)\}$ such that $ (u'_i - v'_i)\equiv  -(u''_i-v''_i)\pmod m$ for all $1\le i\le (m-1)/2$. Suppose for some $i\ne j,\ P_i=P_j$.
WLOG, assume $(u'_i - v'_i)\equiv (u'_j-v'_j)\pmod m$. Consequently, for a pair $(x,y)\in \tilde S$, we have $\{u'_i n+x,v'_in+y\}$ and $\{u'_jn+x,v'_jn+y\}$ are in $W^X_{ST}$. 

But these pairs of $W^X_{ST}$ yield the same differences modulo $mn$. Hence, $W^X_{ST}$ is not a starter. The rest of the proof is analogous to that of Theorem \ref{W is a starter}.
\end{proof}

Thus, Theorems \ref{W is a partition} and \ref{W is a starter} are particular cases of Theorems \ref{WX is a partition} and \ref{WX is a starter} respectively, where $X$  coincides with $\bar T\cup \bar T'$. The direct part of the statement of Theorem \ref{W is a starter} was proven in \cite{b28} as the subtractive and Skolem nucleus can be viewed as a special case of a \textit{perfect difference matrix}.

\begin{theorem}\label{WXST is strong and skew} The set $W^X_{ST}$ (Definitions \ref{X} and \ref{WXST}) is a strong (skew) 2-partition of $\ZZ_{mn}^*$ if and only if $X_m$ is skew in $\ZZ_m$, $S$ is strong (skew) 2-partition of $\ZZ_n^*$ and $T$ is a strong (skew) 2-partition of $\ZZ_m^*$.
\end{theorem}
\begin{proof} 

Here, in order to obtain a strong 2-partition $W^X_{ST}$, we do not require the second 2-partition $T$ to be skew (unlike in Theorem \ref{W is strong}), because now the pairs of type ($i_X$)  are formed from the skew nucleus $X_m$ and the strong first 2-partition $S$.
The rest of the proof is analogous to those of Theorems \ref{W is strong} and \ref{skew theorem}.
\end{proof}

\begin{theorem}\label{XSkolemness} The set $W^X_{ST}$ (Definitions \ref{X} and \ref{WXST}) is a Skolem 2-partition of $\ZZ_{mn}^*$ if and only if $X_m$ is Skolem in $\ZZ_m$ and $S$ and $T$ are Skolem 2-partitions of in $\ZZ_n^*$ and $\ZZ_m^*$ respectively.
\end{theorem}
\begin{proof} 
The proof is analogous to that of Theorem \ref{Skolemness}.
\end{proof}

Note that a nucleus $X_m$, which is both skew and subtractive, does not exist for some odd integer orders $m\le 3$. This can be shown using the notion of a \textit{strong permutation} of a set of  elements in a group.

\begin{remark} Recall that a permutation $\pi$ is called strong if the maps $i\mapsto (\pi(i)-i)$ and $i\mapsto (\pi(i)+ i)$ are permutations, too. In 1973, Wallis and Mullin proved \cite{b25} that if $G$ is a group of odd order $n,\ 3\mid n$, and the 3-Sylow subgroup of $G$ is cyclic, then $G$ does not admit a strong permutation. We adjust this statement to our context.
\end{remark}
\begin{lemma}\label{strong permutation} A skew and subtractive nucleus $X_m$ of order $m\ge 3$ exists if and only if $3\nmid m$.
\end{lemma}
\begin{proof} (a) According to Definitions \ref{X} and \ref{Xproperties}, the existence of a skew and subtractive $X_m$ is equivalent to the existence of a strong permutation $i\mapsto\pi(i)$ of a set $\{0,1,...,m-1\}$, given $0\mapsto\pi(0)=0$:
\begin{equation}
\pi:\ v_i\mapsto u_i,\ 1\le i\le m-1.
\end{equation}

For the sake of contradiction, assume  that $3\mid m$. Then we can write $m=3^tk,\ t\ge 1, \ 3\nmid k$.

Assuming that $\pi$ is a strong permutation of the elements of $\ZZ_m$, consider the sums:
\begin{equation}
\begin{split}
\sum_{i\in\ZZ_m^*}i^2\equiv\sum_{i\in\ZZ_m^*}\pi(i)^2  & \equiv\sum_{i\in\ZZ_m^*}\pi(i)^2+\sum_{i\in\ZZ_m^*}i^2+2\sum_{i\in\ZZ_m^*}i\pi(i) \\&
\equiv\sum_{i\in\ZZ_m^*}\pi(i)^2+\sum_{i\in\ZZ_m^*}i^2-2\sum_{i\in\ZZ_m^*}i\pi(i) \pmod m.
\end{split}
\end{equation}

Taking the second line from the first one, we get $\sum_{i\in\ZZ_m^*}i\pi(i)\equiv 0 \pmod m$. Hence, $\sum_{i\in\ZZ_m^*}i^2\equiv 0 \pmod m$.

But since $k,\ (2\cdot 3^tk-1),\ \frac{3^tk-1}{2}$ are all coprime to $3$, we have $$\sum_{i=1}^{3^tk-1}i^2=\frac{3^tk(3^tk-1)(2\cdot 3^tk-1)}{6}=\frac{3^{t-1}k(2\cdot 3^tk-1)(3^tk-1)}{2}\not\equiv 0\pmod{m}.$$ This is a contradiction. Thus, $3\nmid m$.

(b) Let $3\nmid m$. Then the permutation $\pi:\ i\mapsto 2i$ of the elements of $\ZZ_m$ is clearly strong. (The use of this  permutation for constructing strong starters was offered by Gross in \cite{b22}.) The permutation yields the skew and subtractive nucleus, which we  denote by $C_m$ 
\begin{equation}\label{Cm}
X_m=C_m=\{(i, 2i\pmod{m}),\ 0\le i\le m-1\}.
\end{equation}
\end{proof}

\begin{remark} A nucleus $X_m$ of order $m$ divisible by 3 may be skew, unless we require that it is also subtractive. For example, we can take $X_9=\bar{Q}\cup\bar{Q'}$ of order 9, where $Q$ is the 2-partition of $\ZZ_9^*$ from Example \ref{example of independence of properties}. In this case, $X_9$ is skew and Skolem, but not subtractive.
\end{remark}

Next subsection discusses the product with the nucleus $C_m$ found in the proof of Lemma \ref{strong permutation}.

\subsection{The cardioidal product}\label{Cardioidal product of two 2-partitions}

\begin{df}\label{dfn of cardioidal product} 
The set $C_m$ (\ref{Cm}) is called a cardioidal nucleus of order $m$. 

For 2-partitions, $S$ of $\ZZ^*_{m},$ and $T$ of $\ZZ^*_{n},\ m=2q+1,\ n=2p+1$, $q,p\ge 1$,
the product introduced by Definition \ref{WXST} with nucleus $X_m=C_m$ is called a cardioidal product of $S$ and $T$. It is denoted by  $W^c_{ST}$. 
\end{df}
\begin{lemma} \label {w=wc}
If $T$ is a cardioidal starter then $W_{ST}=W^c_{ST}$. 
\end{lemma}
\begin{proof} If $T$ is a cardioidal starter (refer to Definition  \ref{strong, skew starters}) then $\bar T\cup\bar T'= \{(i,2i\pmod{m}), 1\le i\le m-1\}$.
The rest follows from definitions \ref{WST}  and \ref{dfn of cardioidal product}.
\end{proof} 

 Note, that if $T$ is a cardioidal 2-partition of $\ZZ_m^*$, but not a starter, Theorem \ref{w=wc} does not work.

 The following theorem clarifies when the cardioidal nucleus obeys the conditions of each theorem in Subsection \ref{The parametrized product of starters}.

\begin{theorem}\label{The cardioidal parameter} The cardioidal nucleus $C_m$ is subtractive and Skolem for all odd $m\ge 3$. The cardioidal nucleus  $C_m$ is skew if and only if $3\nmid m$.
\end{theorem}
\begin{proof} The statement follows from Lemma 3.2 of \cite{b21}, and Lemma \ref{strong permutation}.
\end{proof}

Then, we have:

\begin{theorem}\label{main1}  If $3\nmid m$ and $S$ in $\ZZ_n$ and $T$ in $\ZZ_m$ are strong (skew) Skolem starters, then so is $W^c_{ST}$.
\end{theorem}
\begin{proof}
The statement follows from Theorems \ref{WXST is strong and skew}, \ref{XSkolemness} and \ref{The cardioidal parameter}.
\end{proof}

Remarkably, a pair of cardiodial starters does not necessarily produce a cardioidal starter. The following example demonstrates this idea.

\begin{ex}\label{Example 121} Consider the starter $R=\{\{1,2\},\{7,9\},\{3,6\},\{4,8\},\{5,10\}\}$ in $\ZZ_{11}$.  It is cardioidal as $2\equiv 1\cdot 2\pmod{11}$, $7\equiv 9\cdot 2\pmod{11}$,  
$6\equiv 3\cdot 2\pmod{11}$, $8\equiv 4\cdot 2\pmod{11}$, $10\equiv 5\cdot 2\pmod{11}$.

But the product $W_{RR}$ of $R$ with itself contains the pair $\{7,9\}$, which is not cardioidal $\pmod{121}$:
$9\cdot 2\equiv 18\not\equiv 7 \pmod{121}$ and $7\cdot 2\equiv 14\not\equiv 9 \pmod{121}$.

Note that in this case $W_{RR}=W^c_{RR}$ by Lemma \ref{w=wc}. 
\end{ex}

Let us indicate all possible cases, when a product $W_{ST}$ of two cardioidal starters, $S$ and $T$, is cardioidal.

\begin{lemma}\label{the cases} A product $W_{ST}$ of two cardioidal starters, $S$ in $\ZZ_n$ and $T$ in $\ZZ_m$, is cardioidal if and only if either $n=m=3$ or $m>3$ and $\tilde S=(1,2)$.
\end{lemma}
\begin{proof}
\begin{enumerate}
\item $n=m=3$. The case $\tilde S=(1,2)$  is given in Example \ref{ex1}. Choosing $\tilde S=(2,1)$ does not change $W_{ST}$.
\item $n=3,\ m>3$ and $\tilde S=(1,2)$. By Definition \ref{WST}, the pairs of type (ii) are cardioidal as $T$ is cardioidal. Consider a pair $\{u,v\}\in W_{ST}$ of type (i). We have $\{u,v\}=\{in+1,jn+2\}$, where $j\equiv 2i\pmod m$. Then $2(in+1)\equiv jn+2\pmod{mn}$. That means that $\{u,v\}$ is cardioidal. Hence, $W_{ST}$ is a cardioidal starter of order $mn$.
\item $n=3,\ m>3$ and $\tilde S=(2,1)$. Then $\{5,7\}\in W_{ST}$. But this pair is not cardioidal of order  $mn>9$. Hence, $W_{ST}$ is not cardioidal.
\item $n> 3$.  Then either $\{-2,-1\}\pmod{n}$ or $\{-4,-2\}\pmod{n}$ is a pair in $S$ as it is cardioidal of order $n$. Regardless of what kind of starter $T$ is, either $\{n-2,n-1\}$ or $\{n-4,n-2\}$ appears in $W_{ST}$. But neither of these pairs is cardioidal of order $mn>9$. Hence, $W_{ST}$ is not cardioidal.  

\end{enumerate}
\end{proof}

Note that all strong Skolem starters referred to in Theorem \ref{Ogandzhanyants theorem1} and constructed in \cite{b21} are cardioidal. The paper \cite{b21} establishes that each strong cardioidal starter is skew, and hence, by Theorems \ref{skew theorem}, the product of strong cardioidal starters is a skew starter. For instance, $W_{RR}$  from Example \ref{Example 121} is a skew starter in $\ZZ_{121}$ (as well as $W_{RR'}$ and $W_{R'R'}$).

Theorem \ref{Ogandzhanyants theorem1} says that $\ZZ_n$ admits a (cardioidal) skew Skolem starter for all $n$ from $\overline{C_2\setminus\{3\}}$, i.e. from the multiplicative closure of the set $C_2\setminus\{3\}$, where $C_2=\{p\ \mbox{prime}| \ord_p(2)\equiv 2\pmod{4}\}$. Although in \cite{b21}, it was proved that skew cardioidal starters do not exist for orders other than indicated in Theorem \ref{Ogandzhanyants theorem1}, it was also shown that $C_2$ is infinite.

Therefore, Theorem \ref{intermediate result} implies an explicit construction of an infinite family of skew Skolem starters of all composite orders from $\overline{C_2\setminus\{3\}}$. This family is fully new because it consists of starters that are not cardioidal. This family, in its turn, gives rise to further explicit construction of infinitely many strong (and even skew) Skolem starters.

Meanwhile, Theorems \ref{intermediate result} and \ref{main} do not limit us to these composite orders. If we find, by some means, a strong Skolem starter $S$ of an order $n\not\in \overline{C_2\setminus\{3\}}$, Theorems \ref{W is strong} and \ref{WXST is strong and skew} pave an explicit way to construct a strong Skolem starter of any order $nm$, where $m\in \overline{C_2\setminus\{3\}}$. The following example illustrates this idea. 

\begin{ex} Consider a strong Skolem starter  $S=\{\{25,26\},\{20,22\},\{21,24\},\{8,12\},\\\{18,23\},\{10,16\},\{7,14\},\{1,9\},\{2,11\},\{3,13\},\{4,15\},\{5,17\},\{6,19\}\}$, found by Shalaby \cite{b09}, and $T=\{\{1,2\},
\{15,17\},\{13,16\},\{4,8\},\{5,10\},\{6,12\},\{7,14\},\{3,11\},\{9,18\}\}$ is skew in $\ZZ_{19}$. Then both $W_{ST}$ and $W_{ST}^c$ are  strong Skolem starters in $\ZZ_{27\cdot19}=\ZZ_{513}$.
In general, we can construct a strong Skolem starter in any $\ZZ_{27m},\ m\in \overline{C_2\setminus\{3\}}$. Hence, we receive infinitely many strong Skolem starters of orders divisible by 3.
\end{ex}

Moreover, by Theorem \ref{main1}, given a strong Skolem starter $S$ in $\ZZ_n$, where
 $3\nmid n$, we can construct a strong Skolem starter of any order $n^tm,\ t\ge 1, m\in \overline{C_2\setminus\{3\}}$. If, in addition,  $S$  is skew, then we can construct a skew Skolem starter of any order $n^tm,\ t\ge 1, m\in \overline{C_2\setminus\{3\}}$. 

Finally, we note a possibility of a product with a non-cardioidal nucleus, as shown in the following example.

\begin{ex}  Let $R=W^c_{PQ}$, where $P$ and $Q$ are cardioidal starters of orders $k$ and $l$ respectively, $\{k,l\}\subset \overline{C_2\setminus\{3\}}$. Then by Theorem \ref{main} and Lemma \ref{the cases}, $R$ is a skew Skolem starter, but not cardioidal. Consider $X=\bar R\cup \bar R'$. Clearly, $X$ is a skew Skolem nucleus of order $n=kl$, and $X$ is not cardioidal.

Let also $S$ and $T$ be strong Skolem starters of orders $m$ and $n$ respectively.  Then by Theorem \ref{main}, $W^X_{ST}$ is a strong Skolem starter of order $mn$. However, $W^X_{ST}$ is not cardioidal. This can be shown by the reasoning similar to that given in Lemma \ref{the cases}, Case 4.

\end{ex}

\section{Conclusion}\label{Conclusion}

In this paper, we introduced  products of two 2-partitions of $\ZZ^*_{n}$ and $\ZZ^*_{m}$ that give a 2-partition of $\ZZ^*_{nm}$. These binary multi-valued operations are interesting by themselves and deserve further investigations. 

The products reveal a remarkable phenomenon: the resulting partition inherits some properties of the initial ones such as being a starter, or being strong, skew or Skolem. Moreover, in many cases, if the resulting partition has these properties then so do the initial ones. Our results, partly relying on findings of Gross \cite{b22}, extend  his results in the special case of cyclic groups $\ZZ_n$. 
\smallskip

Not every property is passed from 2-partitions to their products. The product $W_{ST}$ of two strong starters, by Theorem \ref{W is strong}, does not lead to a strong starter unless one of the initial starters is skew. However, the product $W_{ST}^X$ of two strong starters is a strong starter if the nucleus $X_m$ is subtractive and skew, for example, if it is cardioidal $C_m$ (\ref{Cm}), where $3\nmid m$. 

Pursuing our particular interest in constructing (strong, skew) Skolem starters, we discuss the reverse side of the useful recursive tools to generate new Skolem starters out of the known ones, and hence, new Skolem sequences. (The latter, as it was briefly explained in Section \ref{introduction}, are undoubtedly valuable combinatorial objects.) 

\begin{remark}\label{enumeration}
For strong (skew) Skolem starters $S$ and $T$, we consider three specific choices of the nucleus $X$ when forming the pairs of strong (skew) Skolem $W^X_{ST}$:
\begin{enumerate}
\item $X=\bar T\cup\bar T'$; 
\item $X$ is cardioidal of the same order as $T$;
\item $X=\bar R\cup\bar R'$, where $R$ is a (skew but not cardioidal) Skolem starter of the same order as $T$.
\end{enumerate}
\end{remark}
Based on Lemma \ref{the cases}, we can state that the main objective of the paper has been achieved: we constructed new family of strong Skolem starters that are not cardioidal.
\begin{theorem}\label{main}There are infinitely many strong and there are infinitely many skew Skolem starters that are not cardioidal.
\end{theorem}
\begin{proof} 
The infinitude of the class of strong (skew) cardioidal starters is proven in \cite{b21}.
The choices of nucleus outlined in 
Remark \ref{enumeration} in the construction of $W^X_{ST}$ based on two strong (skew) cardioidal starters lead to a new strong (skew) Skolem starter which is not cardioidal. 
\end{proof}
Our new results give further support to Shalaby's conjecture stated in 1991.
\smallskip

Observe that the approach of constructing new strong, skew and Skolem starters outlined in this paper, has its limitations because not every starter in a group $\ZZ_{mn}$, is a product of starters in $\ZZ_n$ and $\ZZ_m$ respectively. For example, due to non-existence of Skolem starters of order 5 and 7 and non-existence of a strong starter of order 5, the strong Skolem starter of order 35 found in \cite{b09} can not be constructed as a product $W^X_{ST}$ of two starters with a nucleus, as it would have contradicted Theorems \ref{XSkolemness} and \ref{WXST is strong and skew}.

Since the theorems of this paper establish equivalence conditions of existence of starters of certain types and orders, it helps to cut off some unsuccessful directions in their search. Finding a successful way in these cases is yet an open problem and a subject of further investigations.


\begin{thebibliography}{}
\bibitem{b27}  Chen, K., Ge G., and Zhu L.: Starters and related codes, J. Statist. Plann. Inference, 86 (2000), 379--395
\bibitem{b014} Colbourn, C.J. and Dinitz, J. H.: Handbook of Combinatorial Design (Second edition). Chapman and Hall/CRC, Boca Raton, FL, 2007
\bibitem{b06} Dinitz, J.H. and Stinson, D.R.: Contemporary design theory: A collection of Surveys. John Wiley \& Sons, Inc., New York, 1992
\bibitem{b07} Dinitz, J.H. and Stinson, D.R.: A fast algoritm for finding strong starters. SIAM J. Alg. Disc. Math., 2:1 (1981), 50--56
\bibitem{b23} Fanceti\`{c}, N. and Mendelsohn, E.: A survey of Skolem-type sequences and Rosa's use of them. Math. Slovaca 59(2009), No. 1, 39--76 

\bibitem{b22} Gross, K.B.: A multiplication theorem for strong starters. Aeq. Math. 11, 169–173 (1974). 
\bibitem{b05} Horton, J.D.: Orthogonal starters in finite abelian groups. Discrete Mathematics, 79 (1989/90), 265--278 
\bibitem{b015} Ireland, K., Rosen, M.: A classical introduction to modern number theory (Second edition). Springer-Verlag, New York Inc., New York, 1990

\bibitem{b03} Lan, L., Chang, Y. and Wang, L.: Construction of cyclic quaternary constant-weight codes of weight three and distance four. Designs, Codes and Cryptography, 86:5  (2018), 1063-1083

\bibitem{b017} Linek, V., Jiang, Z.: Extended Langford Sequences with Small Defects. Journal of Combinatorial  Theory,  Series  A, 84 (1998), 38--54  
\bibitem{b013} Linek, V., Mor, S., Shalaby, N.: Skolem and Rosa rectangles and related designs. Discrete Mathematics, 331 (2014), 53--73

\bibitem{b01} Mullin, R.C. and Nemeth, E.: An existence theorem for Room squares. Canad. Math. Bull., 12 (1969),  493--497

\bibitem{b016} Mullin, R.C., Stanton, R. G.: Construction of Room Squares. Ann. Math. Statist., 39:5 (1968),  1540--1548

\bibitem{b21} O. Ogandzhanyants, M. Kondratieva and N. Shalaby: Strong Skolem starters, J. Combin. Des.27(2018), no. 1, 5–21
\bibitem{b09} Shalaby, N.: Skolem sequences: generalizations and applications. Thesis (PhD). McMaster University (Canada), 1991

\bibitem{b04} Simpson J.E.: Langford sequences: perfect and hooked. Discrete Mathematics, 44 (1983), 97--104 

\bibitem{b02} Skolem, T.: On certain distributions of integers in pairs with given differences. Mathematica Scandinavica, 5 (1957), 57--68
\bibitem{b28} Turgeon, J. M. : An upper bound for the length of additive sequences of permutations, Utilitas
Math 17 (1980), 189--196
\bibitem{b26} V\'{a}zquez-\'{A}vila, A.: On strong Skolem starters (2020), https://arxiv.org/pdf/1907.05266.pdf
\bibitem{b25} Wallis W.D., Mullin R.C.: Recent advances on complementary and Room squares, Proc. 4th Southeastern Conference on Combinatorics, Graph Theory and Computing (1973), 521--531 
\end{thebibliography}
\end{document}